\documentclass[12pt]{amsart}

\usepackage{amssymb,amsmath,amsthm}

\def\NN{{\mathbb N}}

\newtheorem{Theorem}{Theorem}[section]
\newtheorem{Lemma}[Theorem]{Lemma}
\newtheorem{Corollary}[Theorem]{Corollary}

\theoremstyle{definition}
\newtheorem{Remark}[Theorem]{Remark}
\newtheorem{Example}[Theorem]{Example}

\begin{document}

\title{Reverse lexicographic Gr\"obner bases and\\
strongly Koszul toric rings}

\author{Kazunori Matsuda}
\address[Kazunori Matsuda]{Department of Mathematics, College of Science, Rikkyo University, 
 Toshima-ku, Tokyo 171-8501, Japan}
\email{matsuda@rikkyo.ac.jp}
\author{Hidefumi Ohsugi}
\address[Hidefumi Ohsugi]{Mathematical Sciences,
Faculty of Science and Technology,
Kwansei Gakuin University,
Sanda, Hyogo, 669-1337, Japan}
\email{ohsugi@kwansei.ac.jp}

\begin{abstract}
Restuccia and Rinaldo proved that a standard graded $K$-algebra $K[x_1,\ldots,x_n]/I$
is strongly Koszul
if the reduced Gr\"obner basis of
$I$ with respect to any reverse lexicographic order 
is quadratic.
In this paper, we give a sufficient condition for a toric ring $K[A]$
to be strongly Koszul in terms of the reverse lexicographic 
Gr\"obner bases of its toric ideal $I_A$.
This is a partial extension of a result given by Restuccia and Rinaldo.
In addition, we show that 
any strongly Koszul toric ring generated by
squarefree monomials is compressed.
Using this fact, we show that our sufficient condition for $K[A]$
to be strongly Koszul is both necessary and sufficient
when $K[A]$ is generated by squarefree monomials.
\end{abstract}

\maketitle

\section*{Introduction}

Herzog, Hibi, and Restuccia \cite{HHR} introduced the notion of strongly Koszul algebras. 
Let $R$ be a standard graded $K$-algebra with
the graded maximal ideal $\mathfrak{m}$.
Then $R$ is said to be {\em strongly Koszul} if 
 $\mathfrak{m}$ admits a minimal system of generators 
$u_1,\ldots, u_n$ of the same degree such that 
for any 
$1 \le i_{1} < \cdots < i_{r} \le n$ and
for all $j = 1, 2, \ldots, r$, 
the colon ideal $(u_{i_{1}}, \ldots, u_{i_{j - 1}}) : u_{i_{j}}$
of $R$ is generated by a subset of elements of $\{u_{1}, \ldots, u_{n}\}$. 
As inspired by this notion, Conca, Trung, and Valla \cite{CTV}
introduced the notion of Koszul filtrations. 
A family $\mathcal{F}$ of ideals of $R$ is called a {\em Koszul filtration} if $\mathcal{F}$ satisfies 
(i) every $I \in \mathcal{F}$ is generated by linear forms;
(ii) $(0)$ and $\mathfrak{m}$ are in $\mathcal{F}$; and 
(iii) for each non-zero ideal $I \in \mathcal{F}$, there exists $J \in \mathcal{F}$ with $J \subset I$ 
such that $I / J$ is cyclic and $J : I \in \mathcal{F}$. 
For example, if $R$ is strongly Koszul, then 
$\mathcal{F} = 
\{ (0) \} \cup \{(u_{i_{1}}, \ldots, u_{i_{r}})
\ |  \  1 \le i_{1} < \cdots < i_{r} \leq n , \ 1 \leq r \leq n\}$
 is a Koszul filtration of $R$. 
The existence of a Koszul filtration of $R$ is an effective sufficient condition for $R$ to be Koszul. 
Some classes of Koszul algebras which have special Koszul filtrations 
have been studied, 
e.g., universally Koszul algebras \cite{C} and initially Koszul algebras \cite{B}. 
On the other hand, 
it is important to characterize the Koszulness in terms of the Gr\"{o}bner bases of its 
defining ideal.
It is a well-known fact that if $R$ is G-quadratic (i.e., its defining ideal has a quadratic Gr\"{o}bner basis) 
then $R$ is Koszul. 
Conca, Rossi, and Valla \cite{CRV} proved that, if $R$ is initially Koszul, then $R$ is G-quadratic. 
Moreover, they and Blum gave a necessary and sufficient condition for 
$R$ to be initially Koszul in terms of initial ideals of toric ideals (\cite{B, CRV}). 

Let $A = \{u_1, \ldots, u_n\}$ be a set 
of monomials of the same degree in a polynomial ring $K[T]=K[t_1,\ldots,t_d]$
 in $d$ variables over a field $K$.  
Then the {\em toric ring} $K[A] \subset K[T]$ is a semigroup ring generated by 
the set $A$ over $K$. 
Let $K[X]= K[x_1,\ldots,x_n]$ be a polynomial ring in $n$ variables over $K$.
The {\em toric ideal} $I_A$ of $K[A]$
is the kernel of the surjective homomorphism $\pi  : K[X] \longrightarrow K[A]$
defined by $\pi (x_i) = u_i$ for each $1 \leq i \leq n$.
Then we have $K[A]  \simeq K[X]/I_{A}$. 
A toric ring $K[A]$ is called {\em compressed} \cite{Sul} if
$\sqrt{{\rm in}_< (I_A)} = {\rm in}_< (I_A)$ 
for any reverse lexicographic order $<$.

In this paper, we study Gr\"obner bases of 
toric ideals of strongly Koszul toric rings. 
First, in Section 1,
we give a sufficient condition for $K[A]$ to be strongly Koszul in terms of 
the Gr\"obner bases of $I_A$ (Theorem \ref{main1}). 
We then have Corollary \ref{revlexquad2}, i.e., 
if the reduced Gr\"obner basis of
$I_A$ with respect to any reverse lexicographic order 
is quadratic, then $K[A]$ is strongly Koszul \cite[Theorem 2.7]{RR}. 
On the other hand,
Examples \ref{cexa1} and \ref{cexa2} are counterexamples of 
\cite[Conjecture 3.11]{RR} (i.e., counterexamples of the converse of Corollary \ref{revlexquad2}).
In Section 2, we discuss strongly Koszul toric rings generated by squarefree monomials.
We show that such toric rings are compressed (Theorem \ref{compressed}). 
Using this fact, we show that the sufficient condition for $K[A]$ to be strongly Koszul in Theorem \ref{main1} 
is both necessary and sufficient
when the toric rings are generated by squarefree monomials (Theorem \ref{converse}).

\section{Gr\"obner bases and strong Koszulness}

First, we give a sufficient condition for toric rings to be strongly Koszul
in terms of the reverse lexicographic Gr\"obner bases.
We need the following lemma:

\begin{Lemma}
\label{key}
Suppose that, for each $1 \leq i < j \leq n$,
there exists a monomial order $\prec$ such that,
with respect to $\prec$,
an arbitrary binomial $g$ in the reduced Gr\"obner basis
of $I_A$ satisfies the following conditions:
\begin{enumerate}
\item[(i)]
$x_i \ | \ {\rm in}_\prec (g)$ and 
$x_j \not| \  {\rm in}_\prec (g)
 \Longrightarrow 
$
$g = x_i x_k - x_j x_\ell$
for some $1 \leq k, \ell \leq n$,
\item[(ii)]
$x_j \ | \ {\rm in}_\prec (g)$ and 
$x_i \not| \  {\rm in}_\prec (g)
 \Longrightarrow$
$g = x_j x_\ell - x_i x_k$
for some $1 \leq k, \ell \leq n$.
\end{enumerate}
Then, $K[A]$ is strongly Koszul.
\end{Lemma}

\begin{proof}
Suppose that $K[A]$ is not strongly Koszul.
By \cite[Proposition 1.4]{HHR}, there exists a monomial $u_{k_1} \cdots u_{k_s}$ of 
a minimal set of generators of $(u_i) \cap (u_j)$
such that $s \geq 3$.
Since $u_{k_1} \cdots u_{k_s}$ belongs to $(u_i) \cap (u_j)$,
there exist binomials 
$x_{k_1} \cdots x_{k_s} - x_i X^\alpha$
and
$x_{k_1} \cdots x_{k_s} - x_j X^\beta$
in $I_A$.
Let ${\mathcal G}$ be the reduced Gr\"obner basis 
of $I_A$ with respect to $\prec$.
Since $ x_i X^\alpha -  x_j X^\beta \in I_A$ is reduced to 0
with respect to ${\mathcal G}$,
it follows that both  $ x_i X^\alpha$ and $ x_j X^\beta$ are reduced to 
the same monomial $m$ with respect to ${\mathcal G}$.
Suppose that $g \in {\mathcal G}$ is used
in the computation $ x_i X^\alpha \xrightarrow{{\mathcal G}}{} m$
and that $x_i$ divides ${\rm in}_\prec (g)$.
If $x_j$ divides ${\rm in}_\prec (g)$, then it follows that
$x_{k_1} \cdots x_{k_s} - x_i x_j X^\gamma$ belongs to $I_A$.
Thus, $u_i u_j$ divides $u_{k_1} \cdots u_{k_s}$.
This contradicts that $u_{k_1} \cdots u_{k_s}$ belongs to
a minimal set of generators of $(u_i) \cap (u_j)$.
If $x_j$ does not divide ${\rm in}_\prec (g)$, then
$g = x_i x_k - x_j x_\ell$ by assumption (i).
Hence, 
$u_i u_k \in (u_i) \cap (u_j)$ divides
$u_{k_1} \cdots u_{k_s}$.
This contradicts that $u_{k_1} \cdots u_{k_s}$ belongs to
a minimal set of generators of $(u_i) \cap (u_j)$.
Therefore, 
$x_i $ never appears in the initial monomials of 
$g \in {\mathcal G}$ which are used 
in the computation $ x_i X^\alpha \xrightarrow{{\mathcal G}}{} m$.
Hence, $x_i$ divides $m$.
By the same argument, it follows that  
$x_j$ never appears in the initial monomials of 
$g \in {\mathcal G}$ which are used 
in the computation $ x_j X^\beta \xrightarrow{{\mathcal G}}{} m$,
and hence, $x_j$ divides $m$.
Thus, $x_i x_j$ divides $m$, which means that $u_i u_j$ divides
$u_{k_1} \cdots u_{k_s}$.
This contradicts that $u_{k_1} \cdots u_{k_s}$ belongs to
a minimal set of generators of $(u_i) \cap (u_j)$.
\end{proof}

Let $G(I)$ denote the (unique) minimal set of monomial generators
of a monomial ideal $I$.
Given an ordering 
$
x_{i_1} < x_{i_2}   < \cdots < x_{i_n}
$ of variables $\{x_1,\ldots,x_n\}$,
let $<_{\rm rlex}$ denote the reverse lexicographic order 
induced by the ordering $<$.

\begin{Theorem}
\label{main1}
Suppose that, for each $1 \leq i < j \leq n$,
there exists an ordering
$
x_{i_1} < x_{i_2}   < \cdots < x_{i_n}
$
with $\{i_1, i_2\} =\{i,j\}$,
such that any monomial in
$G({\rm in}_{<_{\rm rlex}} (I_A) ) \cap ( x_{i_2} )$
is quadratic.
Then, $K[A]$ is strongly Koszul.
\end{Theorem}

\begin{proof}
We may assume that $x_j < x_i$.
By Lemma \ref{key},
it is enough to show that $<_{\rm rlex}$ satisfies conditions (i) and (ii) 
in Lemma \ref{key}.
Let $g$ be an arbitrary (irreducible) binomial in the reduced Gr\"obner basis
of $I_A$ with respect to $<_{\rm rlex}$.

Since $x_j$ is the smallest variable,
$x_j$ does not divide ${\rm in}_{<_{\rm rlex}} (g)$.
Hence,  $<_{\rm rlex}$ satisfies condition (ii).
Suppose that $x_i$ divides ${\rm in}_{<_{\rm rlex}} (g)$.
By the assumption for $<$, $\deg ({\rm in}_{<_{\rm rlex}} (g)) =2$. 
Hence, 
$g = x_i x_p - x_q x_r$ for some $1 \leq p,q,r \leq n$.
Since $x_q x_r  <_{\rm rlex}  x_i x_p  $,
we have $j \in \{q,r\}$, and hence, $<_{\rm rlex}$ satisfies condition (i).
\end{proof}

As a corollary, in case of toric rings,
we have a result of 
Restuccia and Rinaldo \cite[Theorem 2.7]{RR}:

\begin{Corollary}
\label{revlexquad2}
Suppose that
the reduced Gr\"obner basis of $I_A$ is 
quadratic
with respect to any reverse lexicographic order.
Then, $K[A]$ is strongly Koszul.
\end{Corollary}

\begin{Example}
Let $K[A_n] = K[s, t_1 s, \ldots, t_n s, t_1^{-1} s, \ldots, t_n^{-1}s]$.
Then, $I_{A_n}$ is the kernel of the surjective homomorphism
$\pi : K[X] \longrightarrow K[A_n]$ defined by
$\pi(z) =s$, $\pi(x_i) = t_i s$, and $\pi(y_i) = t_i^{-1} s$. 
It is easy to see that $K[A_n]$ is isomorphic to 
$$K[A_G^{\pm}]=
K[s, t_1t_{n+1} s, \ldots, t_n t_{n+1} s, t_1^{-1} t_{n+1}^{-1} s, \ldots, t_n^{-1} t_{n+1}^{-1}s],
$$ where 
$A_G^{\pm}$ is the {\em centrally symmetric configuration}  \cite{central} of $A_G$
associated with the star graph $G=K_{1,n}$ with $n+1$ vertices.
By \cite[Theorem 4.4]{central}, $I_{A_n}$ is generated by
${\mathcal F} =\{ x_i y _i -z^2 \ | \ i = 1,2, \ldots, n\}.$
Then, the Buchberger criterion tells us that
the set
$
{\mathcal F} \cup \{ x_i y _i - x_j y_j \ | \ 1 \leq i < j \leq  n\}
$
is a Gr\"obner basis  of $I_{A_n}$ with respect to 
any monomial order (i.e., a universal Gr\"obner basis of $I_{A_n}$).
Thus, by Corollary \ref{revlexquad2},
$K[A_n]$ is strongly Koszul for all $n \in \NN$.
Eliminating the variable $z$ from ${\mathcal F}$,
by the same argument above,
it follows that the toric ring
$K[B_n] = K[t_1 s, \ldots, t_n s, t_1^{-1} s, \ldots, t_n^{-1}s]$
is strongly Koszul for all $n \in \NN$.
Note that $K[B_n]$ is isomorphic to some toric ring generated by
squarefree monomials.
\end{Example}

\begin{Remark}
A standard graded $K$-algebra $R$ is said to be
{\em $c$-universally Koszul} \cite{EHH}
if the set of all ideals of $R$ which are generated by subsets of the variables
is a Koszul filtration of $R$.
Ene, Herzog, and Hibi proved that
a toric ring $K[A]$ is 
$c$-universally Koszul if 
the reduced Gr\"obner basis of $I_A$ is 
quadratic
with respect to any reverse lexicographic order
\cite[Corollary 1.4]{EHH}.
However, it is known that a toric ring
$K[A]$ is $c$-universally Koszul if and only if 
$K[A]$ is strongly Koszul.
See \cite[Definition 7.2]{Peeva} or
\cite[Lemma 3.18]{Murai}.
So, \cite[Corollary 1.4]{EHH} is equivalent to 
Corollary \ref{revlexquad2}.
\end{Remark}

In Section 2, we will show that 
the converse of Theorem \ref{main1}
holds when $K[A]$ is generated by squarefree monomials.
However, the converse does not hold in general.

\begin{Example}
\label{cexa1}
It is known \cite{HHR} that any Veronese subring of a polynomial
ring is strongly Koszul.
Let $K[A]$ be the fourth Veronese subring of 
$K[t_1,t_2]$,
i.e.,
$K[A] = K[t_1^4, t_1^3 t_2, t_1^2 t_2^2, t_1 t_2^3, t_2^4]$.
Then $I_A$ is generated by the binomials
$$
x_3x_5 - x_4^2,\ 
x_1x_3 - x_2^2,\ 
x_3^2 - x_2x_4,\ 
x_1x_5 -x_2x_4 ,\ 
x_2x_3 - x_1x_4,\ 
x_3x_4 - x_2x_5.
$$
Let $<$ be an ordering of variables such that
$
x_{i_1} < x_{i_2}   < x_{i_3} < x_{i_4} < x_{i_5}
$
with $\{ i_1, i_2\} = \{2,4\}$.
Since both $x_2^3 - x_1^2x_4$ and $x_4^3-x_2x_5^2$
belong to $I_A$, 
it is easy to see that either $x_2^3$ or $x_4^3$ 
belongs to $G({\rm in}_{<_{\rm rlex}} (I_A) ) \cap ( x_{i_2} )$.
Thus, $I_A$ does not satisfy the hypothesis of Theorem \ref{main1}.
\end{Example}

On the other hand,
the converse of Corollary \ref{revlexquad2}
does not hold even if $K[A]$ is generated by squarefree monomials.
Note that
Examples \ref{cexa1} and \ref{cexa2} are counterexamples of 
\cite[Conjecture 3.11]{RR}.

\begin{Example}
\label{cexa2}
Let $K[A] = K[t_4, t_1 t_4, t_2 t_4, t_3 t_4, t_1t_2 t_4, t_2t_3 t_4, t_1t_3 t_4, t_1t_2t_3t_4]$,
which is the toric ring of the stable set polytope 
of the empty graph with three vertices.
Since any empty graph is {\em trivially perfect}
(see also Example \ref{stableset}),
$K[A]$ is strongly Koszul.
See \cite{Mat} for the details.
The toric ideal $I_A$ is generated by the binomials
$$
x_1x_5-x_2x_3, \ 
x_1x_6 -x_3x_4, \ 
x_1x_7 - x_2x_4, \ 
x_5x_6 - x_3x_8,\ 
x_6x_7 - x_4x_8,\ 
x_5x_7 - x_2x_8, 
$$
$$
x_1x_8 -x_4x_5,\ 
x_2x_6 - x_4x_5,\ 
x_3x_7 - x_4x_5.
$$
Let $<$ be an ordering $x_4 < x_3 < x_2 < x_1 < x_8 < x_7 < x_6 < x_5$.
Since, with respect to $<_{\rm rlex}$,
the initial monomial (i.e., the first monomial) of any quadratic binomial above
does not divide the initial
monomial $x_2 x _3 x_8$ of 
$x_4x_5^2 - x_2 x _3 x_8 \in I_A$,
we have
$x_2 x _3 x_8 \in G({\rm in}_{<_{\rm rlex}}(I_A))$.
Thus, 
the reduced Gr\"obner basis of $I_A$
with respect to $<_{\rm rlex}$ is not quadratic.
Below, we show that Theorem \ref{converse} guarantees
that $I_A$ satisfies the hypothesis of Theorem \ref{main1}.
See also Example \ref{stableset}.
\end{Example}

\section{Strongly Koszul toric rings generated by squarefree monomials}

In this section, we consider the case when $K[A]$ is {\em squarefree},
i.e., $K[A]$ is isomorphic to a semigroup ring generated by squarefree monomials.
A toric ring $K[A]$ is called {\em compressed} \cite{Sul} if
$\sqrt{{\rm in}_< (I_A)} = {\rm in}_< (I_A)$ 
for any reverse lexicographic order $<$.
It is known that $K[A]$ is normal if it is compressed.

\begin{Theorem}
\label{compressed}
Suppose that $K[A]$ is  strongly Koszul.
Then, the following conditions are equivalent:
\begin{enumerate}
\item[(i)]
$K[A]$ is squarefree;
\item[(ii)]
$I_A$ has no quadratic binomial
of the form $x_i^2 - x_j x_k$;
\item[(iii)]
$K[A]$ is compressed.
\end{enumerate}
In particular, any squarefree strongly Koszul toric ring is compressed.
\end{Theorem}

\begin{proof}
First, (i) $\Longrightarrow $ (ii) is trivial.
By \cite[Theorem 2.4]{Sul}, we have (iii) $\Longrightarrow $ (i).
Thus it is enough to show (ii) $\Longrightarrow $ (iii).

Let $K[A]$ be a strongly Koszul toric ring
such that $I_A$ has no quadratic binomial
of the form $x_i^2 - x_j x_k$.
Suppose that an irreducible binomial
$f = x_i^2 X^\alpha - x_j X^\beta$ belongs to the reduced Gr\"obner basis
of $I_A$ with respect to a reverse lexicographic order $<_{\rm rlex}$  and
that $x_j$ is the smallest variable in $f$.
Then, $u_i^2 U^\alpha $ belongs to $(U^\alpha) \cap (u_j)$.
Since $K[A]$ is strongly Koszul,
by \cite[Corollary 1.5]{HHR},
$(U^\alpha) \cap (u_j)$ is generated by 
the element in  $(U^\alpha) \cap (u_j)$ of degree $\leq \deg (X^\alpha) +1 $.
Hence, $u_i^2 U^\alpha $ is generated by such elements.
Thus, there exist binomials
$x_i^2 X^\alpha - X^\alpha x_k x_\ell$
and
$x_i^2 X^\alpha - x_j X^\gamma x_\ell$
in $I_A$.
Then, we have $ x_i^2 -  x_k x_\ell \in I_A$.
By assumption, we have $
x_i^2 -  x_k x_\ell =0$, and hence, $i = k = \ell$.
Thus, the binomial $g= x_i X^\alpha - x_j X^\gamma$ belongs to $I_A$.
Since $x_j$ is the smallest variable in $f$,
it follows that $x_i X^\alpha $ is the initial monomial of $g$.
This contradicts that $f$ appears in the reduced Gr\"obner basis of $I_A$
with respect to $<_{\rm rlex}$.
Hence, $K[A]$ is compressed.
\end{proof}

\begin{Example}
\label{stableset}
Let $G$ be a simple graph on the vertex set $V(G) = \{1, \ldots, d\}$
with the edge set $E(G)$.
A subset $S \subset V(G)$ is said to be {\em stable}
if $\{i,j\} \notin E(G)$ for all $i,j \in S$.
For each stable set $S$ of $G$, we define
the monomial $u_S = t_{d+1} \prod_{i \in S} t_i$
in $K[t_1,\ldots, t_{d+1}]$.
Then the toric ring $K[Q_G]$ generated by 
$\{u_S \ | \ S \mbox{ is a stable set of } G \}$
over a field $K$ is called the toric ring of
the {\em stable set polytope} of $G$.
It is known that
\begin{itemize}
\item
$K[Q_G]$ is compressed $\Longleftrightarrow $
$G$ is perfect
(\cite[Example 1.3 (c)]{OHcompressed}, \cite{Thomas}).
\item
$K[Q_G]$ is strongly Koszul $\Longleftrightarrow $
$G$ is trivially perfect (\cite[Theorem 5.1]{Mat}).
\end{itemize}
Here, a graph $G$ is said to be {\em perfect} if the size of maximal clique of $G_{W}$ 
equals to the chromatic number of $G_{W}$ for any induced subgraph $G_{W}$ of $G$, 
and a graph $G$ is said to be {\em trivially perfect} if the size of maximal stable set of $G_{W}$ 
equals to the number of maximal cliques of $G_{W}$ for any induced subgraph $G_{W}$ of $G$.
(About the standard terminologies of graph theory, see \cite{Dies}.)
Since any trivially perfect graph is perfect \cite{G}, 
these facts are consistent with Theorem \ref{compressed}.
On the other hand, with respect to some {\em lexicographic order},
the initial ideal of the toric ideal in Example \ref{cexa2}
is not generated by squarefree monomials.
\end{Example}

Using Theorem \ref{compressed}, we now show that
the converse of Theorem \ref{main1}
holds when $K[A]$ is squarefree:

\begin{Theorem}
\label{converse}
Suppose that $K[A]$ is squarefree and strongly Koszul.
Let $1 \leq i < j \leq n$,
and let $<$ be any ordering of variables satisfying
$$
x_j <  x_i  < \{ x_k \ | \  u_i u_k / u_j \in K[A]  , k \neq j \} < \mbox{ other variables}.
$$
Then, any monomial in $G({\rm in}_{<_{\rm rlex}} (I_A) ) \cap ( x_i )$
is quadratic.
\end{Theorem}

\begin{proof}
Let ${\mathcal G}$ be the reduced Gr\"obner basis of $I_A$
with respect to $<_{\rm rlex}$.
Suppose that 
$
x_i X^\alpha \in 
G({\rm in}_{<_{\rm rlex}} (I_A) ) \cap ( x_i )
$ is not quadratic.
Then, there exists a binomial $g= x_i X^\alpha  - x_j X^\beta$ in ${\mathcal G}$.
Note that ${\rm in}_{<_{\rm rlex}} (g) = x_i X^\alpha $ is squarefree 
by Theorem \ref{compressed}.
Hence, $X^\alpha $ is not divisible by $x_i$.
Moreover, since ${\mathcal G}$ is reduced, $X^\alpha $ is not divisible by $x_j$.

Since $g$ belongs to $I_A$, 
it follows that $u_i U^\alpha = u_j U^\beta$ belongs to
the ideal $(u_i) \cap (u_j)$.
Then, $u_i U^\alpha$ is generated by $u_i u_k = u_j u_\ell \in (u_i) \cap (u_j)$
for some $1 \leq k, \ell \leq n$.
Thus, there exist binomials
$x_i X^\alpha - x_i x_k X^\gamma$
and
$x_i X^\alpha - x_j x_\ell X^\gamma$
in $I_A$.
Then, we have 
$ X^\alpha - x_k X^\gamma \in I_A.$
If $k \in \{i,j \}$, then 
$ X^\alpha  \in {\rm in}_{<_{\rm rlex}} (I_A) $.
This contradicts $
x_i X^\alpha \in 
G({\rm in}_{<_{\rm rlex}} (I_A) ) $.
Hence, $k \notin \{i,j\}$.
Then, $0 \neq x_i x_k - x_j x_\ell \in I_A$
and $ {\rm in}_{<_{\rm rlex}} ( x_i x_k - x_j x_\ell ) = x_i x_k$.
Since $x_i X^\alpha$ is not divisible by $x_i x_k$,
$X^\alpha$ is not divisible by $x_k$.
In particular, $0 \neq  X^\alpha - x_k X^\gamma \in I_A$
and $X^\alpha <_{\rm rlex} x_k X^\gamma$.
Since $u_i u_k / u_j  = u_\ell$, $x_k$ belongs to 
$\{ x_k \ | \  u_i u_k / u_j \in K[A]  , k \neq j \} $.
Thus, the smallest variable $x_m$ appearing in $X^\alpha$
belongs to 
$\{ x_k \ | \  u_i u_k / u_j \in K[A]  , k \neq j \} $.
Let $u_{m'} = u_i u_m / u_j$.
Then, $x_i x_m - x_j x_{m'} $ $(\neq 0)$ belongs to $I_A$.
Therefore, $x_i x_m$ belongs to $ {\rm in}_{<_{\rm rlex}} (I_A) $
and divides $x_i X^\alpha$, which is a contradiction. 
\end{proof}

By Theorem \ref{converse},
we can check whether
a squarefree toric ring $K[A] = K[u_1,\ldots,u_n]$ is strongly Koszul
by computing the reverse lexicographic Gr\"obner bases of $I_A$
at most $n(n-1)/2$ times.

\bigskip
\bigskip

\noindent
{\bf Acknowledgement.}
This research was supported by the JST (Japan Science and Technology Agency)
CREST (Core Research for Evolutional Science and Technology) research project
Harmony of Gr\"{o}bner Bases and the Modern Industrial Society, within the framework of the
JST Mathematics Program ``Alliance for Breakthrough between Mathematics and
Sciences."

\end{document}